\documentclass{amsart}

\theoremstyle{plain}
\newtheorem{theorem}{Theorem}[section]
\newtheorem{proposition}[theorem]{Proposition}

\theoremstyle{definition}

\theoremstyle{remark}
\newtheorem{remark}[theorem]{Remark}

\numberwithin{equation}{section}
\allowdisplaybreaks[4]

\newcommand{\abs}[1]{\lvert#1\rvert}

\newcommand{\sn}{{\mathbb S}}
\newcommand{\rn}{{\mathbb R}}

\newcommand{\hn}{{\mathbb H}}
\renewcommand{\c}{{\mathbb C}}
\renewcommand{\i}{\sqrt{-1}}
\newcommand{\stc}{{\mathbb S}^{2}(c)}
\newcommand{\SU}{\mathrm{SU}(2)}

\newcommand{\SO}{\mathrm{SO}(3)}

\newcommand{\su}{\mathfrak{su}(2)}
\newcommand{\h}{{\mathcal H}}
\renewcommand{\v}{{\mathcal V}}

\DeclareMathOperator{\spa}{Span} 
 
\DeclareMathOperator{\ad}{ad}

\begin{document}

\title[Cheeger-Gromoll metrics and Hopf map]
{Generalized Cheeger-Gromoll Metrics \\
and the Hopf map}

\author{Michele Benyounes, Eric Loubeau}
\address{D{\'e}partement de Math{\'e}matiques \\
Universit{\'e} de Bretagne Occidentale \\
6  avenue Victor Le Gorgeu \\
CS 93837, 29238 Brest Cedex 3 \\ France}
\email{Michele.Benyounes@univ-brest.fr, 
Eric.Loubeau@univ-brest.fr}

\author{Seiki Nishikawa}
\address{Mathematical Institute \\ 
Tohoku University \\
Sendai, 980-8578 \\
Japan}
\email{nisikawa@math.tohoku.ac.jp}
\thanks{The third author was supported in part by the 
Grant-in-Aid for Scientific Research (B) of the Japan Society 
for the Promotion of Science, No.\ 20340009. }



\dedicatory{Dedicated to Professor Udo Simon on his 
seventieth birthday}


\begin{abstract}
We show, using two different approaches, that there exists 
a family of Riemannian metrics on the tangent bundle of 
a two-sphere, which induces metrics of constant curvature 
on its unit tangent bundle.
In other words, given such a metric on the tangent bundle
of a two-sphere, the Hopf map is identified with a Riemannian 
submersion from the universal covering space of the unit 
tangent bundle, equipped with the induced metric,
onto the two-sphere.
A hyperbolic counterpart dealing with the tangent bundle 
of a hyperbolic plane is also presented.
\end{abstract}

\maketitle

\section{Introduction}\label{sect:1}
One of the most studied maps in Differential Geometry is 
the Hopf map $H : \sn^{3} \to \c {\mathrm P}^{1}$
from the unit three-sphere $\sn^{3} \subset \c^{2}$ 
onto the complex projective line 
$\c {\mathrm P}^{1} = \c \cup \{\infty\}$, 
defined for $z = (z_{1}, z_{2}) \in \sn^{3}$ by
\begin{equation*}
 H(z) =
 \begin{cases}
  z_{1}/z_{2}  &  \mbox{if $z_{2} \neq 0$},  \\
  \infty  &  \mbox{if $z_{2} = 0$}.
 \end{cases}
\end{equation*}
Composed with the inverse stereographic projection
$p^{-1} : \c \to \sn^{2} \setminus \{(0, 0, 1)\} \subset
\mathbb{R}^{3}$ given by 
\begin{equation*}
 p^{-1}(\zeta) = \left( \frac{2\operatorname{Re}\zeta}
 {|\zeta|^{2} + 1}, \frac{2\operatorname{Im}\zeta}
 {|\zeta|^{2} + 1}, \frac{|\zeta|^{2} - 1}{|\zeta|^{2} + 1}\right),
 \quad  \zeta \in \mathbb{C},
\end{equation*}
it can be regarded as a map $H : \sn^{3} \to \sn^{2}$
sending
\begin{equation}\label{eqn:1.1}
 z = (z_{1}, z_{2})  \mapsto  
 \left(2\operatorname{Re}z_{1}\bar{z}_{2},
 2\operatorname{Im}z_{1}\bar{z}_{2},
 |z_{1}|^{2} - |z_{2}|^{2}\right),
\end{equation}
which, if we choose the two-sphere $\sn^{2}$ to be of radius 
$1/2$, becomes a Riemannian submersion, 
relative to the canonical metric on each sphere.

As is well known, the Hopf map is closely linked to the unit 
tangent bundle $T^{1}\sn^{2} \to \sn^{2}$ of the two-sphere.
Indeed, the total space $T^{1}\sn^{2}$ is diffeomorphic to 
the real projective three-space $\rn \mathrm{P}^{3}$,
and the Hopf map $H : \sn^{3} \to \sn^{2}$ is nothing else 
than the canonical projection from the universal covering space 
of $T^{1}\sn^{2}$ onto $\sn^{2}$.
This shows that a Riemannian metric of constant positive
curvature exists on $T^{1}\sn^{2}$, inherited from the 
canonical metric on $\sn^{3}$.

Then it is a pertinent question whether this constant curvature 
metric on $T^{1}\sn^{2}$ is induced from some ``natural'' 
Riemannian metric defined on the ``ambient'' total space $T\sn^{2}$
of the tangent bundle $T\sn^{2} \to \sn^{2}$ of $\sn^{2}$, 
when one regards the total space of the unit tangent bundle 
$T^{1}\sn^{2}$ as a hypersurface of $T\sn^{2}$.
This question also arises when the three-sphere $\sn^{3}$ is 
equipped with one of the Berger metrics, that is, 
when a homothety is applied on the fibres.

The aim of this paper is to give affirmative answers, 
using generalized Cheeger-Gromoll metrics $h_{m,r}$ defined in
\cite{BLW1} (see Subsection~\ref{sect:3.3} for the precise
definition of $h_{m,r}$), that there is a two-parameter family of 
Riemannian metrics on the tangent bundle of $\sn^{2}$, 
which induces desired metrics for both questions.
Namely, we prove the following 
\begin{theorem}\label{thm:1}
Let \,$\sn^{n}(c)$ be the $n$-sphere of constant curvature $c > 0$, 
and denote by $T\sn^{n}(c)$ $($resp.\ $T^{1}\sn^{n}(c)$$)$ 
its tangent $($resp.\ unit tangent\/$)$ bundle.  
Let $F : \sn^{3}(c/4) \to T^{1}\stc$ be the covering map defined 
by \eqref{eqn:2.8}.

$(1)$  Then $F$ induces an isometry from the projective three-space
$(\rn \mathrm{P}^{3}(c/4), \,g_{\mathrm{can}})$ of 
constant curvature $c/4$ to $T^{1}\stc$, equipped with 
the metric induced from the generalized Cheeger-Gromoll metric 
$h_{m,r}$ on $T \sn^{2}(c)$, 
where $m = \log_{2}c$ and $r \geq 0$.

$(2)$  Similarly, when \,$\sn^{3}$ is equipped with a Berger metric
$g_{\epsilon}$ defined by \eqref{eqn:3.10}, 
$F$ induces an isometry from
$(\rn \mathrm{P}^{3}, \,g_{\epsilon})$ to 
$(T^{1}\sn^{2}(4), \,h_{m,r})$, 
for $m = \log_{2}{\epsilon^{2}} + 2$
and $r \geq 0$.
\end{theorem}

In particular, we see from Theorem~\ref{thm:1}\,(1) that 
any three-sphere of constant positive curvature
is isometrically immersed into the total space of the tangent
bundle of a two-sphere, equipped with a generalized 
Cheeger-Gromoll metric.
A hyperbolic counterpart of this is also true.
Namely, any anti-de Sitter three-space of constant negative
curvature is isometrically immersed into the total space of
the tangent bundle of a hyperbolic plane, equipped 
with an indefinite generalized Cheeger-Gromoll metric.
More precisely, we prove
\begin{theorem}\label{thm:2}
Let $H^{3}_{1}(c)$ be the anti-de Sitter three-space of
constant curvature $-c < 0$.
Let $T\mathbb{H}^{2}(c)$ 
$($resp.\ $T^{1}\mathbb{H}^{2}(c)$$)$
be the tangent $($resp. unit tangent\/$)$ bundle of the
hyperbolic plane $\mathbb{H}^{2}(c)$ of constant curvature
$-c < 0$, and endow $T\mathbb{H}^{2}(c)$ with the indefinite
generalized Cheeger-Gromoll metric $h_{m,r}$ defined by 
\eqref{eqn:5.14}.
Then the covering map $F : H^{3}_{1}(c/4) \to 
T^{1}\mathbb{H}^{2}(c)$ defined by \eqref{eqn:5.8}
is an isometric immersion from $H^{3}_{1}(c/4)$ to
$T^{1}\mathbb{H}^{2}(c)$, equipped with the metric induced
from $h_{m,r}$, where $m = \log_{2}c$ and $r \geq 0$.
\end{theorem}

The paper is organized as follows.
In Section~\ref{sect:2} we describe the Hopf map 
$\mathbb{S}^{3}(c/4) \to \mathbb{S}^{2}(c)$
in terms of the natural identification of the three-sphere
$\mathbb{S}^{3}(c/4)$ and the unit tangent bundle
$T^{1}\mathbb{S}^{2}(c)$ with Lie groups 
$\mathrm{SU}(2)$ and $\mathrm{SO}(3)$, respectively.
Then, using these descriptions, we prove Theorem~\ref{thm:1}
in Section~\ref{sect:3}.
For this end, we compute the differential of the covering map
$F : \mathbb{S}^{3}(c/4) \to T^{1}\mathbb{S}^{2}(c)$
and find explicitly a suitable induced metric on 
$T^{1}\mathbb{S}^{2}(c)$ making $F$ to be isometric.
An alternative proof of Theorem~\ref{thm:1}, based on 
our previous knowledge of the curvature of generalized 
Cheeger-Gromoll metrics, is presented in Section~\ref{sect:4}.

In Section~\ref{sect:5} we prove a hyperbolic counterpart
of Theorem~\ref{thm:1}\,(1).
Namely, we define the hyperbolic Hopf map
$H^{3}_{1}(c/4) \to \mathbb{H}^{2}(c)$
for the hyperbolic plane, and extend the notion of generalized 
Cheeger-Gromoll metrics to admit indefinite ones.
Then we prove Theorem~\ref{thm:2} by the same method
as in Section~\ref{sect:3}, namely, by identifying the anti-de Sitter 
three-space $H^{3}_{1}(c/4)$ and the unit tangent bundle 
$T^{1}\mathbb{H}^{2}(c)$ with Lie groups 
$\mathrm{SU}(1, 1)$ and $\mathrm{SO}^{+}(1, 2)$,
respectively.

\section{Hopf map}\label{sect:2}
To fix our notation and conventions, 
we first review how one can identify the Hopf map 
$H : \sn^{3} \to \sn^{2}$ with the canonical projection from 
the universal covering space of the unit tangent bundle
$T^{1}\sn^{2}$ onto the $2$-sphere $\sn^{2}$.

To begin with, recall that the unit $3$-sphere
\[
 \sn^3 = \left\{ (x^{1}, x^{2}, x^{3}, x^{4}) \in \rn^{4}
 \mid  (x^{1})^{2} + (x^{2})^{2} + (x^{3})^{2} + 
 (x^{4})^{2} = 1 \right\}
\]
is diffeomorphic to the special unitary group 
\begin{align*}
 \SU   & =  \left\{A \in \mathrm{GL}(2, \c ) \mid
  {}^{t}\bar{A}A  = \mathrm{Id}, \ \det{A} = 1\right\} \\
 & = \left\{
   \begin{pmatrix}
    a  &  -\bar{b}\, \\
    b  &  \bar{a}
  \end{pmatrix}
 \biggm| a, b \in \c, \ \abs{a}^{2} + \abs{b}^{2} =1 \right\}
\end{align*}
under the map
\begin{equation}\label{eqn:2.1}
\begin{aligned}
 \psi : \sn^{3}  & \to \SU, \\
 x = (x^{1}, x^{2}, x^{3}, x^{4})  & \mapsto 
 A_{x} =
  \begin{pmatrix}
  z_{1}  &  - \bar{z}_{2} \\
  z_{2}  &  \bar{z}_{1}
  \end{pmatrix},  
\end{aligned}
\end{equation}
where $z_{1} = x^{1} + \sqrt{-1}x^{2}$ and
$z_{2} = x^{3} + \sqrt{-1}x^{4}$.

Moreover, $\SU$ is the universal covering space of 
the special orthogonal group $\SO$ with the covering map 
\begin{equation*}
 \rho : \SU \to \SO,  \quad  A_{x} \mapsto \rho(A_{x})
\end{equation*}
described as follows.
First, we regard $\SO$ as $\mathrm{SO}(\su)$, 
where the Lie algebra of $\SU$, 
\begin{equation*}
\begin{aligned}
 \su & = \left\{ X \in \mathfrak{gl}(2, \mathbb{C})
 \mid {}^{t}X + \bar{X} = 0, \ \operatorname{Tr}X = 0
 \right\}  \\
 & = \left\{
   \begin{pmatrix}
    \sqrt{-1}x^{3}  &  -x^{2} + \sqrt{-1}x^{1}  \\[0.1cm]
    x^{2} + \sqrt{-1}x^{1}  &  -\sqrt{-1}x^{3}
  \end{pmatrix}
 \biggm| x^{1}, x^{2}, x^{3} \in \mathbb{R} \right\},
\end{aligned}
\end{equation*}
is identified 
with $\mathbb{R}^{3}$, equipped with 
the scalar product 
$\langle X , Y \rangle = - (1/2)\operatorname{Tr}(XY)$, 
so that
\begin{equation}\label{eqn:2.2}
 e_1 = 
  \begin{pmatrix}
  0  &  \i\, \\
   \i  &  0
  \end{pmatrix},  \quad
 e_2 = 
  \begin{pmatrix}
   0 &  - 1 \\
   1 &  0
  \end{pmatrix},  \quad
 e_3 = 
  \begin{pmatrix}
    \i  &  0 \\
   0  &  -\i
  \end{pmatrix}
\end{equation}
form an orthonormal basis of 
$\left(\su, \langle\ , \,\rangle\right)$.
Then $\rho(A_{x})$ is defined by the adjoint representation of
$\SU$ as
\begin{equation}\label{eqn:2.3}
 \rho(A_{x}) :  \su \to \su ,  \quad 
 Y \mapsto \operatorname{Ad}(A_{x})Y =
 A_{x} Y A_{x}^{-1},
\end{equation}
and so $\rho(A_{x}) \in \SO \cong 
\mathrm{SO}\big(\su, \langle\ , \,\rangle\big)$.

The matrix representation of $\rho(A_{x})$, 
with respect to the orthonormal basis \eqref{eqn:2.2} of $\su$,
is given by
\begin{equation}\label{eqn:2.4}
\begin{aligned}
\rho(A_{x})
 & =
 \begin{pmatrix}
      \operatorname{Re}(z_{1}^{2} - \bar{z}_{2}^{2}) &
      \operatorname{Im}(\bar{z}_{1}^{2} + z_{2}^{2})  &
    2 \operatorname{Re}(z_1 \bar{z}_2) \\[0.1cm]
      \operatorname{Im}(z_{1}^{2} - \bar{z}_{2}^{2})  &
      \operatorname{Re}(\bar{z}_{1}^{2} + z_{2}^{2})  &
    2 \operatorname{Im}(z_1 \bar{z}_2) \\[0.1cm]
    - 2 \operatorname{Re}(z_{1}z_{2}) &
    2 \operatorname{Im}(z_{1}z_{2})  &
      \abs{z_{1}}^2 - \abs{z_{2}}^2  
 \end{pmatrix}   \\[0.1cm]
 & = 
 \begin{pmatrix}
  A_{x}e_{1}A_{x}^{-1}  &  A_{x}e_{2}A_{x}^{-1}  &
  A_{x}e_{3}A_{x}^{-1}
 \end{pmatrix}.
 \end{aligned} 
 \end{equation}
Note that $\rho : \SU \to \SO$ is a homomorphism with
kernel $\{\pm \mathrm{Id}\}$, and hence $\SO$ is 
diffeomorphic to the real projective three-space 
$\rn \mathrm{P}^{3}$. 

Given $c > 0$, let $\sn^{n}(c) \subset \rn^{n+1}$ denote 
the $n$-sphere of radius $1/\sqrt{c}$ with center at 
the origin of $\rn^{n+1}$.
We also denote the unit $n$-sphere $\sn^{n}(1)$
simply by $\sn^{n}$.
It should be noted that, with the metric induced from the 
Euclidean metric of $\rn^{n+1}$, $\sn^{n}(c)$ is a space of 
constant positive curvature $c$.

Now, recall that the unit vectors tangent to $\stc$
form the {\em unit tangent bundle}
\begin{equation}\label{eqn:2.5}
\begin{aligned}
 T^{1}\stc  & =  \left\{ (x, v) \in \rn^{3}\times \rn^{3} 
  \mid  x \in \stc, \ v \in T_{x}\stc, \ \abs{v} = 1\right\}  \\
 & =  \left\{ (x, v) \in \rn^{3}\times \rn^{3} \mid
  \abs{x} = 1/\sqrt{c}, \ \abs{v} = 1, \ \langle x, v \rangle = 0
  \right\}
\end{aligned}
\end{equation}
of $\stc$ with the canonical projection $\pi : T^{1}\stc \to \stc$
given by $\pi(x, v) = x$.
Since $T^{1}\stc$ is composed of orthogonal vectors of $\rn^3$, 
one can define the diffeomorphism
\begin{equation}\label{eqn:2.6}
 \phi : \SO \to T^{1}\stc , \quad  (c_{1}\ \,c_{2}\ \,c_{3}) 
 \mapsto  (c_{3}/\sqrt{c}, \,c_{1}). 
\end{equation}
Finally, let $\iota$ be the homothety defined by
\begin{equation}\label{eqn:2.7}
 \iota : \sn^{3}(c/4) \to \sn^{3}(1),  \quad
 2x/\sqrt{c} \mapsto x.  
\end{equation}

Then we have the following
\begin{proposition}
The composition of the covering map
\begin{equation}\label{eqn:2.8}
 F = \phi \circ \rho \circ \psi \circ \iota : \sn^{3}(c/4)
 \to T^{1}\stc  
\end{equation}
with the canonical projection $\pi : T^{1}\stc \to \stc$
is identical with the Hopf map $H : \sn^3(c/4) \to \stc$.
\end{proposition}

Indeed, from \eqref{eqn:2.1} through \eqref{eqn:2.7}, 
we see that the composition $\pi \circ F$ is a map sending
\begin{equation*}
 (2/\sqrt{c})(z_{1}, z_{2})  \mapsto  (1/\sqrt{c})
 \left(2z_1 \bar{z}_2 , |z_{1}|^{2} - |z_{2}|^{2}\right),  
\end{equation*}
which is nothing but the {\em Hopf map} $H$ of \eqref{eqn:1.1} 
normalized in our context.

\section{Differential approach}\label{sect:3}
The most direct path to an answer to our problem is 
to compute the differential of the covering map 
$F: \sn^3(c/4) \to T^{1}\stc$, 
determine the image of an orthonormal frame of $T\sn^3(c/4)$, 
and then find explicitly a suitable induced metric on $T^{1}\stc$ 
making $F$ to be isometric.
This can be carried out as follows.

\subsection{Differentials of maps}\label{sect:3.1}
1) The map $\psi : \sn^{3} \to \SU$ in \eqref{eqn:2.1}
gives to to  a linear map from $\rn^{4}$ into the space of
complex $2 \times 2$ matrices of the form
$\begin{pmatrix}
    a  &  -\bar{b}\, \\
    b  &  \bar{a}
  \end{pmatrix}
$, 
so that $d\psi_{x} = \psi$ for all $x \in \rn^{4}$.

Noting that the fibres of the Hopf map \eqref{eqn:1.1} are
described as the orbits of the $\sn^{1}$-action 
$\sn^{1} \times \sn^{3} \to \sn^{3}$ on $\sn^{3}$
defined by
\begin{equation*}
\big(e^{\i t} , (z_{1}, z_{2})\big) \mapsto 
e^{\i t} (z_{1}, z_{2}) = 
\big(e^{\i t}z_{1}, e^{\i t}z_{2} \big),
\end{equation*}
we see that if $x = (x^{1}, x^{2}, x^{3}, x^{4}) \in \sn^{3}$, 
then 
\[
 X_{3}(x) = (\i z_{1}, \i z_{2}) = (- x^{2}, x^{1}, - x^{4}, x^{3})
\]
is a vector tangent to a fibre of the Hopf map
$H : \sn^{3}(1) \to \sn^{2}(4)$, and
\begin{equation*}
 X_{3}(x),  \quad
 X_{2}(x) = (- x^{3}, x^{4}, x^{1}, - x^{2}),  \quad 
 X_{1}(x) = (- x^{4}, - x^{3}, x^{2}, x^{1})
\end{equation*}
form a global orthonormal frame of $T\sn^3$. 
Since
$\psi(x) = A_{x} =
  \begin{pmatrix}
  z_{1}  &  - \bar{z}_{2} \\
  z_{2}  &  \bar{z}_{1}
 \end{pmatrix}$,
it follows that
\[
 d\psi_{x} = \psi : T_{x} \sn^{3} \to T_{\psi(x)}(\SU)
  = A_{x} \cdot \su
\]
and
\begin{equation}\label{eqn:3.1}
 d\psi_{x}(X_{3}(x)) =
 \begin{pmatrix}
  - x^{2} + \i x^{1} & x^{4} + \i x^{3} \\[0.1cm]
  - x^{4} + \i x^{3} & - x^{2} - \i x^{1}
 \end{pmatrix}
 = A_{x} e_3 .
\end{equation}
Similarly, we have $d\psi_{x}(X_{2}(x)) = A_{x} e_2$ and 
$d\psi_{x}(X_{1}(x)) = A_{x} e_1$.

2) The differential of the covering map
\begin{equation*}
 \rho :  \SU \to \SO,  \quad  A_{x} \mapsto \rho(A_{x}),
\end{equation*}
given by \eqref{eqn:2.3}, is a linear map
\begin{equation*}
 d\rho_{A_{x}} :  T_{A_{x}}(\SU) = A_{x} \cdot \su 
 \to T_{\rho(A_{x})}\SO
 = \rho(A_{x}) \cdot \mathfrak{so}(3)
\end{equation*}
defined by
\begin{equation}\label{eqn:3.2}
 A_{x} Y \mapsto d\rho_{A_{x}}(A_{x} Y)
 = \rho(A_{x}) \circ\ad(Y),
\end{equation}
where
\begin{equation*}
 \ad(Y) : \su \to \su,  \quad  Z \mapsto \ad(Y)(Z) = [Y, Z].
\end{equation*}

Consequently, for the orthonormal basis \eqref{eqn:2.2} of
$\mathfrak{su}(2)$, we obtain, for instance,
\begin{align*}
 d\rho_{A_{x}} : A_{x} \cdot \su  & \to \rho(A_{x}) \cdot 
 \mathfrak{so}(3),   \quad 
 A_{x} e_{3} \mapsto \rho(A_{x}) \circ \ad(e_{3}),
\end{align*}
and $\ad(e_{3})(e_{3}) =0$, $\ad(e_{3})(e_{2}) =-2e_{1}$, 
$\ad(e_{3})(e_{1}) =2e_{2}$.
Therefore, as a matrix,
\begin{equation*}
\ad(e_{3}) =
 \begin{pmatrix}
  0  &  -2  & 0 \\
  2  &  0  & 0 \\
  0  &  0  & 0
 \end{pmatrix},
\end{equation*}
and
\begin{equation*}
 \rho(A_{x}) \circ \ad(e_{3}) = 
 \begin{pmatrix}
  2A_{x} e_{2} A_{x}^{-1}  &  -2A_{x} e_{1} A_{x}^{-1}  &  0
 \end{pmatrix}.
\end{equation*}
Similarly, since $\ad(e_{2})(e_{1}) = -2e_{3}$, we obtain
\begin{align*} 
 \rho(A_{x})\circ \ad(e_{2})  & = 
 \begin{pmatrix}
  -2A_{x} e_{3} A_{x}^{-1}  &  0  &  2 A_{x} e_{1} A_{x}^{-1}
 \end{pmatrix},  \\
 \rho(A_{x})\circ \ad(e_{1})  & = 
 \begin{pmatrix}
  0  &  2A_{x} e_{3} A_{x}^{-1}  &  -2A_{x} e_{2} A_{x}^{-1}
\end{pmatrix}.
\end{align*}

3) Finally, we note that the diffeomorphism
\begin{equation*} 
 \phi : \SO \to T^{1}\stc  
\end{equation*}
defined by \eqref{eqn:2.6} is linear, so $d\phi_{g} = \phi$ 
and, for $\rho(A_{x}) \in \SO$
\begin{equation*}
 d\phi_{\rho(A_{x})} = 
 \phi : T_{\rho(A_{x})}\SO = 
  \rho(A_{x}) \cdot \mathfrak{so}(3)
  \to T_{\phi(\rho(A_{x}))}(T^{1}\stc)
\end{equation*}
is given by
\begin{equation*}
 \left(\alpha_{1} \ \,\alpha_{2} \ \,\alpha_{3}\right)
 \mapsto (\alpha _{3}/\sqrt{c}, \,\alpha_{1}).
\end{equation*}
Therefore we obtain
\begin{equation}\label{eqn:3.3}
\begin{aligned}
 d\phi_{\rho(A_{x})}(\rho(A_{x})\circ \ad(e_{3})) 
 & = (0, \,2A_{x} e_{2} A_{x}^{-1}) = \tilde{e}_{3},  \\
 d\phi_{\rho(A_{x})}(\rho(A_{x}) \circ \ad(e_{2})) 
 & = (2A_{x} e_{1} A_{x}^{-1}/\sqrt{c}, 
 \,-2A_{x} e_{3} A_{x}^{-1}) = \tilde{e}_{2},  \\
 d\phi_{\rho(A_{x})}(\rho(A_{x})\circ \ad(e_{1})) 
 & = (-2A_{x} e_{2} A_{x}^{-1}/\sqrt{c}, \,0)
 = \tilde{e}_{1}.
 \end{aligned}
\end{equation}

\subsection{Lifts to the unit tangent bundle}\label{sect:3.2}
In general, each tangent space of the tangent bundle $TM$ of
a Riemannian manifold $(M, g)$ admits a canonical decomposition
into its vertical and horizontal subspaces.
Indeed, given a point $(x, e) \in TM$,
the kernel of the differential of the canonical projection 
$\pi : TM \to M$ defines the vertical space
$\v_{(x,e)} = \ker d\pi_{(x,e)}$,
while the horizontal space $\h_{(x,e)}$ is given by the kernel of 
the connection map
\[
 K_{(x,e)} =K : T_{(x,e)}TM \to T_{x} M ,  \quad 
 K(Z) = d (\exp_{x} \circ R_{-e} \circ\tau)(Z).
\]
Here $\tau : U \subset TM \to T_{x} M$ is the map,
defined on an open neighbourhood $U$ of $(x, e) \in TM$,
sending a vector $v\in T_{y} M$, with $(y,v) \in U$, 
to a vector in $T_{x} M$ by parallel transport along 
the unique geodesic arc from $y$ to $x$.
The map $R_{-e} :T_{x}M \to T_{x}M$ is the translation 
given by $R_{-e}(X) = X - e$ for $X \in T_{x} M$. 

One can see that $\h_{(x,e)}\cap \v_{(x,e)} = \{0\}$ 
and $\h_{(x,e)} \oplus \v_{(x,e)}=T_{(x,e)}TM$, and 
define the {\em horizontal lift} $X^h \in\h_{(x, e)}$ and 
the {\em vertical lift} $X^{v} \in\v_{(x, e)}$ of 
$X \in T_{x} M$ by
\[ 
 K_{(x,e)}(X^{v}) = X,  \quad  
 d\pi_{(x,e)}(X^{h} ) = X.
\]
An alternative description of the horizontal lift $X^{h}$ is 
given as follows.
Let $X \in T_{x} M$ and choose $e\in T_{x} M$. 
Take a curve $\gamma : I \to M$ such that
$\gamma(0) = x$ and $\dot{\gamma }(0) = X$.
(Since the result is independent of the curve chosen, 
we can take it to be a geodesic.) 
Let $\Gamma : I \to TM$ be a unique curve in $TM$ such that 
$\Gamma(0) = (x, e)$ and $\Gamma(t)$ is parallel to 
$\dot{\gamma}(t)$ in the sense that 
$\nabla_{\dot{\gamma }(t)}\Gamma(t) = 0$ for all $t\in I$.
Namely, $\Gamma(t) = (\gamma(t), v(t))$, where 
$v(t) \in T_{\gamma(t)}M$ and
$\nabla_{\dot{\gamma }(t)}v(t) = 0$ for all $t\in I$, 
so that $v(t)$ is the parallel transport of the vector $e$ 
along the curve $\gamma$. 
Then $\dot{\Gamma}(0) = X^{h} \in T_{(x,e)}TM$. 
We will use this approach below.

Now, recall that the unit tangent bundle $T^{1} \stc$ is 
a $3$-dimensional hypersurface of $T \stc$.
Then we note that at $(x, e) \in T^{1} \stc$ the tangent 
space of the tangent bundle $T \stc$ is written as
\[
 T_{(x,e)}(T \stc) = \big\{ X^{h} + Y^{v} \mid 
 X, Y \in T_{x} \stc \big\},
\]
where $X^{h}$ (resp.\ $Y^{v}$) is the horizontal 
(resp.\ vertical) lift of $X$ (resp.\ $Y$).
Also, that of the unit tangent bundle $T^{1}\mathbb{S}^{2}(c)$
is given by
\begin{equation}\label{eqn:3.4}
 T_{(x,e)}(T^{1} \stc) = \big\{ X^{h} + Y^{v} \mid 
 X, Y \in T_{x} \stc, \ \langle Y, e \rangle = 0 \big\},
\end{equation}
since the tangent vector at $(x, e)$ of 
any vertical curve on $T^{1} \stc$ must be orthogonal 
to $e$.

For the covering map $F: \sn^3(c/4) \to T^1 \stc$, 
we obtain from \eqref{eqn:2.8} together with \eqref{eqn:3.1}
through \eqref{eqn:3.3} that
\begin{equation}\label{eqn:3.5}
 dF_{x} (2X_{3}(x)/\sqrt{c}) = \tilde{e}_{3},  \quad
 dF_{x} (2X_{2}(x)/\sqrt{c}) = \tilde{e}_{2},  \quad
 dF_{x} (2X_{1}(x)/\sqrt{c}) = \tilde{e}_{1},
\end{equation}
and recall that 
\[
 F(2x/\sqrt{c}) = (\tilde{x}, e) \in T^{1} \stc
\]
for each $2x/\sqrt{c} \in \sn^{3}(c/4)$,
where $\tilde{x} = (1/\sqrt{c}) A_{x}e_{3}A_{x}^{-1}$ 
and $e = A_{x}e_{1} A_{x}^{-1}$.
We set
\[
 f = - A_{x}e_{2}A_{x}^{-1}.
\]
Then $( \tilde{x}, f) \in T^{1}\stc$ and 
$\langle f, e \rangle = 0$,
so that, by virtue of \eqref{eqn:3.4},
\[
 T_{(\tilde{x}, e)}(T^{1} \stc) 
 = \spa \big\{ e^{h}, f^{h}, f^{v} \big\}.
\]

Now, we are going to show
\begin{proposition}\label{prop:2}
Let \,$\tilde{x}$, $e$ and $f$ be as above.  Then
\begin{equation}\label{eqn:3.6}
 (\sqrt{c}/2) \tilde{e}_{2}= e^{h},  \quad  
 (\sqrt{c}/2)\tilde{e}_{1}=f^{h},  \quad  
 \tilde{e}_{3} = -2 f^{v}.
\end{equation}
\end{proposition}
\begin{proof}
To construct the horizontal lift 
$e^{h} \in T_{(\tilde{x}, e)}(T^{1} \stc)$, 
we take the great circle $\gamma$ in $\stc$ such that
$\gamma(0)=\tilde{x}$ and $\dot{\gamma }(0)= e$, 
that is,
\[
 \gamma(t) = \cos(\sqrt{c}\,t)\,\tilde{x} + 
 \sin(\sqrt{c}\,t)({e}/{\sqrt{c}}).
\]
Then the curve $\Gamma : I \to T^{1} \stc$ given by 
$\Gamma(t) = (\gamma(t), \dot{\gamma}(t))$ is parallel to 
$\dot{\gamma}(t)$, so that $e^{h} = \dot{\Gamma}(0) =
(\dot{\gamma}(0),\ddot{\gamma}(0))$.
Namely,
\[
 e^{h} 
 = (A_{x} e_{1} A_{x}^{-1}, \,-\sqrt{c}A_{x}e_{3}A_{x}^{-1}) 
 = (\sqrt{c}/2)\tilde{e}_{2}.
\]

Similarly, to construct $f^{h} \in T_{(\tilde{x}, e)}(T^{1} \stc)$ 
for $f = - A_{x}e_{2}A_{x}^{-1}$, 
we take the great circle $\gamma(t) = 
\cos(\sqrt{c}\,t)\,\tilde{x} + \sin(\sqrt{c}\,t)({f}/\sqrt{c})$, 
so that $\gamma(0)=\tilde{x}$ and $\dot{\gamma }(0) = f$. 
Then the curve $\Gamma : I \to T^{1} \stc$ given by 
$\Gamma(t) = (\gamma(t), v(t) = e)$ satisfies 
$\nabla_{\dot{\gamma}(t)}v(t) = 0$ for all $t \in I$. 
Hence
\[
 f^{h} =\dot{\Gamma}(0) = (f, 0) = (- A_{x}e_{2}A_{x}^{-1}, \,0) 
 = (\sqrt{c}/2)\tilde{e}_{1}.
\]

Finally, since $d\pi(\tilde{e}_{3}) = 0$, 
to show that $\tilde{e}_{3} = -2f^{v}$ we compute $K(\tilde{e}_{3})$. 
Since $\tilde{e}_{3} = dF_{x} (2X_{3}/\sqrt{c})$ and 
$X_{3} = \dot{\gamma}(0)$ for $\gamma(t) = e^{\sqrt{-1}t}x$,
which is indeed a geodesic of $\sn^{3}$ along a fibre of the Hopf map, 
we can write $\tilde{e}_{3}$ as a vector tangent to a curve
$\tilde{\gamma}(t) = F\circ(2/\sqrt{c})\gamma(t)$
in $T^{1}\sn^{2}(c)$ and then
\begin{equation}\label{eqn:3.7}
 K(\tilde{e}_{3}) = \left.\frac{d}{dt}\right|_{t=0} 
 (\exp_{\tilde{x}} \circ R_{-e} \circ \tau)
 (\tilde{\gamma}(t)).
\end{equation}
Also, it is immediate from \eqref{eqn:2.4} and \eqref{eqn:2.6} that
\[
 \tilde{\gamma}(t) = \big(({1}/{\sqrt{c}}) A_{x}e_{3}A_{x}^{-1}, 
 \,A_{\gamma(t)} e_{1} A_{\gamma(t)}^{-1} \big)
 \in T^{1}\sn^{2}(c)
\]
and $\pi(\tilde{\gamma}(t)) = \tilde{x}$, so that
$\tilde{\gamma}(t)$ is a curve along the fibre over $\tilde{x}$.
Consequently, the parallel transport $\tau$ in \eqref{eqn:3.7} 
is the identity map, and 
\begin{equation*}
 K(\tilde{e}_{3}) = \left.\frac{d}{dt}\right|_{t=0}
 \left(\exp_{\tilde{x}}\left(\frac{1}{\sqrt{c}} 
 A_{x} e_{3} A_{x}^{-1},
 \,A_{\gamma(t)} e_{1} A_{\gamma(t)}^{-1}
  - A_{x} e_{1} A_{x}^{-1} \right)
 \right),
\end{equation*}
since $e = A_{x}e_{1}A_{x}^{-1}$.

Put $W(t) = A_{\gamma(t)} e_{1} A_{\gamma(t)}^{-1}
- A_{x} e_{1} A_{x}^{-1}$. 
Then the geodesic of $\stc$ starting at $\tilde{x}$ with initial 
vector $W(t)$ is given by
\begin{equation*}
 \delta_{t}(s) = \frac{1}{\sqrt{c}} A_{x}e_{1}A_{x}^{-1}
 \,\cos(\sqrt{c}\,\abs{W(t)}s) 
 + \frac{1}{\sqrt{c}} \frac{W(t)}{\abs{W(t)}}
 \sin(\sqrt{c}\,\abs{W(t)}s),
\end{equation*}
and $K(\tilde{e}_{3}) = ({d}/{dt})|_{t=0}\,\delta_{t}(1)$. 
On the other hand, since
\begin{equation*}
 \gamma(t) = (x^{1}\cos t - x^{2}\sin t, x^{2}\cos t 
 + x^{1}\sin t, x^{3}\cos t - x^{4}\sin t, x^{4}\cos t 
 + x^{3}\sin t),
\end{equation*}
we have
\begin{align*}
 W(t)  & = A_{\gamma(t)}e_{1}A_{\gamma(t)}^{-1}
  - A_{x}e_{1}A_{x}^{-1}   \\
 & = \begin{pmatrix}
 - 4(-x^1 x^3 +x^2x^4)\sin^2 t + 2(x^1x^4 + x^2x^3)
 \sin 2t   \\[0.1cm]
 - 4(x^1 x^2 +x^3x^4)\sin^2 t + ((x^1)^2 - (x^2)^2
  + (x^3)^2 - (x^4)^2)\sin 2t  \\[0.1cm]
 - 2((x^1)^2 - (x^2)^2 - (x^3)^2 + (x^4)^2)\sin^2 t 
  - 2(x^1 x^2 -x^3x^4)\sin 2t
\end{pmatrix}
\end{align*}
and $\abs{W(t)} = 2 \sin t$.

Therefore we obtain
\begin{align*}
 K(\tilde{e}_{3})   & =  \left.\frac{d}{dt}\right|_{t=0} 
 \left(\frac{1}{\sqrt{c}} A_{x} e_{3}A_{x}^{-1}
 \,\cos(2\sqrt{c}\sin t)
 + \frac{W(t)}{2\sqrt{c} \sin t} \sin(2\sqrt{c}\sin t)
 \right),  \\[0.2cm]
 & =  \left(\frac{W(t)}{2\sqrt{c}\sin t}\right)(0) 
 \left.\frac{d}{dt}\right|_{t=0} \sin(2\sqrt{c}\sin t)
 \\[0.2cm]
 & = 
 \begin{pmatrix}
  4 (x^{1}x^{4} + x^{2}x^{3}) \\[0.1cm]
  2 ((x^{1})^{2} - (x^{2})^{2} + (x^{3})^{2}
  - (x^{4})^{2}) \\[0.1cm]
  - 4 (x^{1} x^{2} - x^{3}x^{4})
 \end{pmatrix}
 = 2 A_{x}e_{2}A_{x}^{-1}
 \\[0.1cm]
 &  = - 2 f,
\end{align*}
which shows that $\tilde{e}_{3} = -2 f^v$.
\end{proof}

\subsection{Generalized Cheeger-Gromoll metrics}\label{sect:3.3}
For the tangent bundle $TM$ of a Riemannian manifold $(M, g)$, 
a {\em natural} Riemannian metric on $TM$, in the sense that 
with respect to which the vertical and horizontal subspaces of 
each tangent space of $TM$ are orthogonal and the canonical 
projection $\pi : TM \to M$ becomes a Riemannian submersion, 
was first defined by Sasaki \cite{Sasaki}.
This metric, now called the Sasaki metric, appears as having the  
simplest possible form, but its geometry is known to be rather 
rigid (cf.\ \cite{BLW1, MT}).
Later on, a more general metric, called the Cheeger-Gromoll 
metric, was given  on $TM$ by Musso and Tricerri \cite{MT},
which has been further generalized in \cite{BLW1} toward 
the discovery of new harmonic sections of Riemannian 
vector bundles.

To be precise, given the two-sphere $\stc$, 
for $m\in \rn$ and $r \geq 0$, 
the {\em generalized Cheeger-Gromoll metric} $h_{m,r}$ on 
the tangent bundle $T\stc$ is defined, 
on each tangent space $T_{(x,e)}(T \stc)$ at 
$(x, e) \in T\stc$, by
\begin{equation}\label{eqn:3.8}
\begin{aligned}
 h_{m,r} (X^{h} , Y^{h} ) & = \langle X , Y \rangle,
 \quad  h_{m,r} (X^{h} , Y^{v} ) = 0,  \\
 h_{m,r} (X^{v} , Y^{v} ) & = \omega^{m} 
 (\langle X , Y \rangle + r \langle X , e \rangle\langle  Y, 
 e \rangle),
\end{aligned}
\end{equation}
where $X, Y \in T_{x}\sn^{2}(c)$
and $\omega = 1/(1+ |e|^{2})$.
In particular, when $(x, e) \in T^{1}\stc$, this metric restricts 
on $T_{(x, e)}(T^{1}\stc)$ to
\begin{equation}\label{eqn:3.9}
\begin{aligned}
 h_{m,r} (X^{h}, Y^{h}) & = \langle X , Y \rangle,
 \quad  h_{m,r} (X^{h}, Y^{v}) = 0,  \\
 h_{m,r} (X^{v}, Y^{v}) & = \frac{1}{2^{m}} 
 \langle X , Y \rangle,
\end{aligned}
\end{equation}
since $\langle Y, e \rangle = 0$ by virtue of \eqref{eqn:3.4}.
Namely, the parameter $r$ disappears if $h_{m,r}$ is
restricted to the unit tangent bundle $T^{1}\stc$.
It should be noted that the original Cheeger-Gromoll metric
corresponds to $m = r = 1$ and the Sasaki metric to
$m = r = 0$.

Now, our Theorem~\ref{thm:1} can be proved as follows.
If we choose $m = \log_{2}c$, then, noting \eqref{eqn:3.5}
and \eqref{eqn:3.6},
we obtain from \eqref{eqn:3.9} that
\begin{align*}
 & h_{m,r}((\sqrt{c}/2)\tilde{e}_{1}, (\sqrt{c}/2)\tilde{e}_{1}) 
 = h_{m,r}(f^{h} , f^{h} ) = \langle f, f \rangle =1,   \\[0.1cm]
 & h_{m,r}((\sqrt{c}/2)\tilde{e}_{2}, (\sqrt{c}/2)\tilde{e}_{2}) 
 = h_{m,r}(e^{h} , e^{h} ) = \langle e, e \rangle = 1,  \\[0.1cm]
 & h_{m,r}((\sqrt{c}/2)\tilde{e}_{1}, (\sqrt{c}/2)\tilde{e}_{2}) 
 = h_{m,r}(f^{h}, e^{h}) = \langle f, e \rangle= 0,   \\[0.1cm]
 & h_{m,r}((\sqrt{c}/2)\tilde{e}_{2}, (\sqrt{c}/2)\tilde{e}_{3}) 
 = - h_{m,r}(e^{h}, \sqrt{c}f^{v}) = 0,  \\[0.1cm]
 & h_{m,r}((\sqrt{c}/2)\tilde{e}_{1}, (\sqrt{c}/2)\tilde{e}_{3}) 
 = - h_{m,r}(f^{h}, \sqrt{c}f^{v}) = 0,
\end{align*}
and
\begin{equation*}
 h_{m,r}((\sqrt{c}/2)\tilde{e}_{3}, (\sqrt{c}/2)\tilde{e}_{3}) 
 = h_{m,r}( -\sqrt{c}f^{v}, -\sqrt{c}f^{v}) 
 = \frac{c}{2^{m}}\langle f , f \rangle = 1 .
\end{equation*}
This shows that $F : \sn^{3}(c/4) \to T^{1} \stc$ defined by 
\eqref{eqn:2.8} induces an isometry from 
$(\rn\mathrm{P}^{3}(c/4), \,g_{\mathrm{can}})$
to $(T^{1} \stc, \,h_{m,r})$ for $m = \log_{2}c$ and any 
$r \geq 0$. 

Moreover, if we equip the unit three-sphere $\sn^{3}$ with 
a {\em Berger metric} $g_{\epsilon}$ in \cite{Berger} such 
that 
\begin{equation}\label{eqn:3.10}
 \mbox{$\{ X_{1}, X_{2}, \epsilon X_{3} \}$ is an orthonormal 
 frame of $T\sn^{3}$},  
\end{equation}
then we see from \eqref{eqn:3.5} that 
$dF_{x} (\epsilon X_{3} ) = \epsilon \tilde{e}_{3}$ and
\begin{equation*}
 h_{m,r} (\epsilon \tilde{e}_{3}, \epsilon \tilde{e}_{3}) = 
 h_{m,r}(-2\epsilon f^{v}, -2\epsilon f^{v}) = 
 \frac{1}{2^m} \langle 2\epsilon f, 2\epsilon f \rangle = 
 \frac{4\epsilon^{2} }{2^{m}}.
\end{equation*}
Therefore, for $m = \log_{2} \epsilon^{2} + 2$, the map
$F : \sn^{3} \to T^{1} \sn^{2}(4)$ yields an isometry from 
$(\rn\mathrm{P}^{3}, g_{\epsilon})$
to $(T^{1} \sn^{2}(4), h_{m,r})$ for any $r \geq 0$. 

\begin{remark}
In Theorem~\ref{thm:1}~(1), if we choose $c = 1$, then $m = 0$.
Thus, for $r = 0$ the generalized Cheeger-Gromoll metric
$h_{0,0}$ defined by \eqref{eqn:3.8} is nothing but the Sasaki 
metric defined on $T\sn^{2}(1)$.  
In this case, Theorem~\ref{thm:1}~(1) is proved in \cite{KS}.
\end{remark}

\section{Curvature approach}\label{sect:4}
To show that $(T^{1}\stc, \,h_{m,r})$ with $m = \log_{2}c$ 
is isometric to $(\rn\mathrm{P}^3(c/4), \,g_{\mathrm{can}})$,
an alternative method  is to compute that both of them have 
constant sectional curvatures $c/4$. 
To carry out this, we regard $T^{1}\stc$ as a hypersurface of 
$T\stc$ and combine the Gauss formula with our previous 
knowledge of the curvature of $(T\stc, \,h_{m,r})$. 

To this end, let $\nabla$ and $R$ denote, respectively, 
the covariant differentiation and the curvature tensor defined on 
$\sn^{2}(c)$, and let $\tilde{\nabla}$ denote the covariant 
differentiation defined on $T\sn^{2}(c)$.
Then standard computation (cf.~\cite{BLW2}) 
shows that the Levi-Civita connection on $(T\stc , \,h_{m,r})$ is 
given by
\begin{align*}
 \tilde{\nabla}_{X^h} Y^h  & = 
  (\nabla_{X} Y)^h -\frac{1}{2} (R(X, Y)e)^v, \\
 \tilde{\nabla}_{X^h} Y^v  & =
  (\nabla_{X} Y)^v +\frac{1}{2} \omega^{m} (R(e, Y)X)^h, \\
 \tilde{\nabla}_{X^v} Y^h  & =
  \frac{1}{2} \omega^{m} (R(e, X)Y)^h, \\
 \tilde{\nabla}_{X^v} Y^v  & = 
  -m \omega [\langle X, e \rangle Y+ \langle Y, e \rangle X]^v 
   + (m \omega+r)\omega_{r} \langle X, Y \rangle U  \\
 & \quad  + mr\omega\omega_{r} \langle X,
  e\rangle \langle Y, e\rangle U 
\end{align*}
for all $X\in T_{x}\stc$, $Y \in C^{\infty}(T\stc)$ and 
$(x, e) \in T\stc$, where $U \in C^{\infty}(T\stc)$ is 
the canonical vertical vector field on $T\stc$ defined by 
$U(x, e) = e^{v}$ and $\omega_{r} = 1/(1+ r |e|^{2})$. 

It should be noted that $\tilde{\nabla}_{X^{v}} Y^{v}$ has
no horizontal part, so the fibres of $T \stc$ are totally geodesic.
The unit normal $\boldsymbol{n}$ to $T^1 \stc$ in $T \stc$ 
at $(x, e)$ is proportional to the canonical vertical vector 
$U(x, e)$, that is, to the vertical lift $e^{v}$ of $e$, and the
normalization factor is $\alpha = \sqrt{{2^{m}}/{(1+r)}}$.

Let $B$ be the second fundamental form of $T^{1} \stc$ in 
$T \stc$.
For $X^{h}$ and $Y^{h}$ in $T_{(x,e)}(T^{1} \stc)$,
\begin{align*}
 B(X^{h}, Y^{h}) & = h_{m,r} \big(\tilde{\nabla}_{X^h} Y^{h} , 
  \,\boldsymbol{n}\big)\boldsymbol{n} 
   = h_{m,r} \left(-\frac{1}{2} (R(X, Y)e)^{v}, 
    \,\alpha e^{v} \right)\boldsymbol{n} \\
 & = -\frac{\alpha}{2^{m+1}}\left[\langle R(X, Y)e, e \rangle
  + r \langle R(X, Y)e, e \rangle\langle e, e \rangle\right] 
   \boldsymbol{n} = 0 .
\end{align*}
For $X^{h}$ and $Y^{v}$ in $T_{(x,e)}(T^{1} \stc)$,
\begin{align*}
 B(X^{h}, Y^{v}) & = h_{m,r} \big(\tilde{\nabla}_{X^{h}} Y^{v}, 
  \,\boldsymbol{n}\big)\boldsymbol{n} 
   = h_{m,r} \left((\nabla_{X} Y)^{v}, \boldsymbol{n}\right)
    \boldsymbol{n} \\
 & = \frac{\alpha}{2^{m}}(1+r) \langle \nabla_{X} Y, 
  e \rangle \boldsymbol{n} = 0,
\end{align*}
since we can extend $Y$ by parallel transport along $X$.
For $X^{v}$ and $Y^{h}$ in $T_{(x,e)}(T^{1} \stc)$,
\begin{equation*}
 B(X^{v}, Y^{h}) = h_{m,r} \big(\tilde{\nabla}_{X^{v}} Y^{h}, 
 \,\boldsymbol{n}\big)\boldsymbol{n} = 0.
\end{equation*}
For $X^{v}$ and $Y^{v}$ in $T_{(x,e)}(T^{1} \stc)$,
\begin{align*}
 B(X^{v}, Y^{v})  & = h_{m,r} \big(\tilde{\nabla}_{X^{v}} Y^{v}, 
  \,\boldsymbol{n}\big)\boldsymbol{n} 
   = h_{m,r} \left(\frac{m/2 +r}{1+r} \langle X , Y \rangle 
   e^{v} , \,\alpha e^{v} \right) \boldsymbol{n} \\
 &  = \alpha \frac{m/2 +r}{1+r} \langle X , Y \rangle 
\boldsymbol{n}.
\end{align*}

From the Gauss formula, the sectional curvature $\hat{K}$
of $(T^{1}\stc, \,h_{m,r})$ can be determined as follows. 
Let $(x,e)\in T^{1} \stc$ and $f\in T_{x}^{1} \stc$ 
such that $\langle e, f\rangle = 0$.
Recall that 
$T_{(x,e)}(T^{1} \stc) = \spa\big\{ e^{h} , f^{h} ,  f^{v}\big\}$.
Then, applying the formulae in \cite[Prop.\ 3.1]{BLW2}, 
we obtain
\begin{proposition}\label{prop:3}
Sectional curvatures of \,$(T^{1}\stc, \,h_{m,r})$ are
given by
\begin{equation}\label{eqn:4.1}
 \begin{aligned}
  \hat{K} (e^{h} \wedge f^{h} )  & = 
   c - \frac{3c^{2}}{2^{m+2}},  \\
  \hat{K} (e^{h} \wedge f^{v} )  & =
   \hat{K} (f^{h} \wedge f^{v} ) = \frac{c^{2}}{2^{m+2}}.
 \end{aligned}
\end{equation}
\end{proposition}
\begin{proof}
Denoting by $\tilde{K}$ the sectional curvature of 
$(T\stc, \,h_{m,r})$, we have
\begin{align*}
 \hat{K} (e^{h} \wedge f^{h} )  & = 
  \tilde{K} (e^{h} \wedge f^{h} )  \\
 &  \quad  + h_{m,r}\big(B(e^{h}, e^{h}), B(f^{h}, f^{h})\big)
  - \abs{B(e^{h}, f^{h})}^2 \\
 & = \tilde{K} (e^{h} \wedge f^{h} ) = 
  c - \frac{3c^{2}}{2^{m+2}},  \\
 \hat{K} (e^{h} \wedge f^{v} )  & = 
  \tilde{K} (e^{h} \wedge f^{v} )  \\
 & \quad  + h_{m,r}\big(B(e^{h}, e^{h}), 
 B(\beta f^{v}, \beta f^{v})\big)
  - \abs{B(e^{h}, \beta f^{v})}^{2} \\
 & = \tilde{K} (e^{h} \wedge f^{v} ) = 
   \frac{c^{2}}{2^{m+2}} , 
\end{align*}
where $\beta f^v$ is of unit length, and
\begin{align*}
 \hat{K} (f^{h} \wedge f^{v})  & = 
  \tilde{K} (f^{h} \wedge f^{v} )  \\
 & \quad  + h_{m,r}\big(B(f^{h}, f^{h}),
 B(\beta f^{v}, \beta f^{v})\big)
  - \abs{B(f^{h},\beta f^{v})}^{2} \\
 & =  \tilde{K} (f^{h} \wedge f^{v} ).
\end{align*}
Note that $\tilde{K} (f^{h} \wedge f^{v})$ cannot be 
computed from the formulae in \cite[Prop.\ 3.1]{BLW2}, 
since $f^{h}$ and $f^{v}$ are the horizontal and vertical lifts 
of the same vector. 
So we compute it as
\begin{equation*}
 \tilde{K} (f^{h} \wedge f^{v})
  = \frac{h_{m,r}(\tilde{R}(f^{h} , f^{v}) f^{v}, f^{h})}
   {|f^{h} \wedge f^{v} |^{2}} = \frac{|R(e,f)f|^{2}}{2^{m+2}} 
  = \frac{c^{2}}{2^{m+2}},
\end{equation*}
$\tilde{R}$ being the curvature tensor on $T\sn^{2}(c)$
(cf.~\cite[Prop.\ 2.3]{BLW2}).
\end{proof}
\begin{remark}
The parameter $r$ in the generalized Cheeger-Gromoll metric 
$h_{m,r}$ has no influence on the sectional curvature of 
$(T^{1}\sn^{2}(c), h_{m,r})$, since it disappears on the 
unit tangent bundle as in \eqref{eqn:3.9}.
On the other hand, it has an effect on the sectional curvature 
of the ambient space $T\sn^{2}(c)$.
For instance, the following are proved in \cite{BLW2}:
\vspace{0.1cm}
\begin{itemize}
 \item[(1)]  $(T\sn^{2}(c), \,h_{m,r})$ has positive sectional 
 curvature if and only if
 \vspace{0.1cm}
  \begin{itemize}
   \item[(i)]  $r = 0, \ m \geq 1, \ \dfrac{4}{3}
    \dfrac{m^{m}}{(m-1)^{m-1}} \geq c > 0$, or
     \vspace{0.1cm}
   \item[(ii)]  $r > 0, \ m = 1, \ 4/3 \geq c > 0$.
  \end{itemize}
   \vspace{0.2cm}
 \item[(2)]  $(T\sn^{2}(c), \,h_{m,r})$ has positive scalar
  curvature if
   \vspace{0.1cm}
   \begin{itemize}
    \item[(i)]  $r = 0, \ 2 \geq m \geq 1, 
     \ 4 \dfrac{m^{m}}{(m-1)^{m-1}} \geq c > 0$, or
      \vspace{0.1cm}
   \item[(ii)]  $r > 0, \ m = 1, \ 4 \geq c > 0$.
  \end{itemize}
\end{itemize}
\end{remark}

An alternative proof of Theorem~\ref{thm:1} 
now goes as follows.
The two values for the sectional curvatures in \eqref{eqn:4.1}
are equal to $c/4$ if $m = \log_{2}c$. 
Hence $(T^{1} \stc, \,h_{m,r})$ 
is isometric to 
$(\rn\mathrm{P}^3(c/4), \,g_{\mathrm{can}})$ 
for any $r \geq 0$.¡¡

Similarly, from formulae in \cite[p.\ 306]{Milnor}, we see that,
when $\sn^3$ is equipped with the Berger metric $g_\epsilon$, 
its sectional curvatures take the values $4 - 3\epsilon^{-2}$ 
and $\epsilon^{-2}$.
Therefore, if we choose
\[
 m = \log_{2}{\epsilon^2} + 2,
\]
then the map $F : \sn^3 \to T^{1}\sn^{2}(4)$ yields 
an isometry from $(\rn\mathrm{P}^3, g_{\epsilon})$
to $(T^{1}\sn^{2}(4), h_{m,r})$ for any $r \geq 0$.

\section{Hyperbolic counterpart}\label{sect:5}
In what follows, we denote by $\mathbb{R}^{n}_{\nu}$
the pseudo-Euclidean $n$-space of index $\nu$,
that is, $\mathbb{R}^{n}$ equipped with the indefinite metric
\[
 \langle x, y \rangle = \sum_{i=1}^{n-\nu}
 x^{i}y^{i} - \sum_{j=n-\nu +1}^{n}
 x^{j}y^{j}.
\]

\subsection{Hyperbolic Hopf map}\label{sect:5.1}
Let $H^{3}_{1}(c)$ be the {\em anti-de Sitter $3$-space}
of constant negative curvature $-c < 0$ (cf.\ \cite{O'Neill}),
which is, by definition, a hypersurface in $\mathbb{R}^{4}_{2}$
defined by $\langle x, x \rangle = -1/c$, that is,
\begin{equation*}
 H^{3}_{1}(c) = \left\{ (x^{1}, x^{2}, x^{3}, x^{4}) \in 
 \mathbb{R}^{4}_{2}  \mid  (x^{1})^{2} + (x^{2})^{2}
 - (x^{3})^{2} - (x^{4})^{2} = -1/c \right\}.
\end{equation*}
Note that $H^{3}_{1}(c)$ is diffeomorphic to 
$\mathbb{S}^{1} \times \mathbb{R}^{2}$.
If we introduce complex coordinates
$z_{1} = x^{1} + \sqrt{-1}x^{2}$ and
$z_{2} = x^{3} + \sqrt{-1}x^{4}$,
then $H^{3}_{1}(c)$ is represented as
\begin{equation*}
 H^{3}_{1}(c) = \left\{ (z_{1}, z_{2}) \in 
 \mathbb{C}^{2} \mid  |z_{1}|^{2} - |z_{2}|^{2} =
 -1/c \right\}.
\end{equation*}

To define the hyperbolic Hopf map, let
$\varpi : \mathbb{C}^{2} \setminus \{0\} \to 
\mathbb{C}\mathrm{P}^{1}$ be the canonical projection
defining the complex projective line $\mathbb{C}\mathrm{P}^{1}$.
Restricting $\varpi$ to $H^{3}_{1}(c) \subset \mathbb{C}^{2}
\setminus \{0\}$, we have a mapping 
\begin{equation*}
 \varpi : H^{3}_{1}(c) \to \mathbb{C},  \quad
 z = (z_{1}, z_{2})  \mapsto  \varpi(z) = z_{1}/z_{2},
\end{equation*}
which maps $H^{3}_{1}(c)$ diffeomorphically onto the unit ball
$B^{2} = \left\{ \zeta \in \mathbb{C}  \mid  |\zeta| < 1 \right\}$ 
in $\mathbb{C}$.
Let
\[
 \mathbb{H}^{2}(c) = \left\{ (x^{1}, x^{2}, x^{3}) \in
 \mathbb{R}^{3}_{1} \mid  (x^{1})^{1} + (x^{2})^{2}
 - (x^{3})^{2} = - 1/c, \,x^{3} > 0 \right\}
\]
be the hyperbolic plane of constant curvature $- c < 0$
embedded in $\mathbb{R}^{3}_{1}$.
Denote by
\begin{equation*}
 p^{-1}(\zeta) = \left( \frac{2\operatorname{Re}\zeta}
 {1 - |\zeta|^{2}}, \frac{2\operatorname{Im}\zeta}
 {1 - |\zeta|^{2}}, \frac{1 + |\zeta|^{2}}{1 - |\zeta|^{2}}\right),
 \phantom{- -}  \quad  \zeta \in B^{2} \subset 
 \mathbb{C},
\end{equation*}
the inverse stereographic projection 
$p^{-1} : B \to \mathbb{H}^{2}(1)$ for $\mathbb{H}^{2}(1)$
from the south pole $(0,0,-1) \in \mathbb{H}^{2}(1)$,
and let $\iota$ be the homothety defined by
\[
 \eta : \mathbb{H}^{2}(1) \to \mathbb{H}^{2}(c),
 \quad  x \mapsto x/\sqrt{c}.
\]
Then, composing $\varpi$ with $\eta \circ p^{-1}$, 
we obtain the {\em hyperbolic Hopf map}
\begin{equation}\label{eqn:5.1}
 H = \eta \circ p^{-1} \circ \varpi : H^{3}_{1}(c/4) \to 
 \mathbb{H}^{2}(c),
\end{equation}
given by
\begin{equation}\label{eqn:5.2}
 H(z) =  (1/\sqrt{c})
 \left(2z_{1}\bar{z}_{2}, |z_{1}|^{2} + |z_{2}|^{2}\right)
 \in \c\times\rn.
\end{equation}

Note that the hyperbolic Hopf map $H$ is a submersion from 
a pseudo-Riemannian manifold $H^{3}_{1}(c/4)$ with 
geodesic fibres, which can be described as the orbits of the 
$\sn^{1}$-action 
$\sn^{1} \times H^{3}_{1}(c/4) \to H^{3}_{1}(c/4)$
on $H^{3}_{1}(c/4)$ defined by
\begin{equation*}
 \big(e^{\i t} , (z_{1}, z_{2})\big) \mapsto 
 e^{\i t} (z_{1}, z_{2}) = 
 \big(e^{\i t}z_{1} , e^{\i t}z_{2} \big).
\end{equation*}
In particular, 
if $x = (x^{1}, x^{2}, x^{3}, x^{4}) \in H^{3}_{1}(1)$,
then
\[
 X_{3}(x) = (\i z_{1} , \i z_{2}) = (- x^{2}, x^{1}, - x^{4}, x^{3})
\]
is a vector tangent to a fibre of the hyperbolic Hopf map
$H : H^{3}_{1}(1) \to \mathbb{H}^{2}(4)$, 
with $\langle X_{3}, X_{3} \rangle = -1$, and
\begin{equation*}
 X_{3}(x),  \quad  
 X_{2}(x) = (x^{3}, - x^{4}, x^{1}, - x^{2}),  \quad
 X_{1}(x) = (x^{4}, x^{3}, x^{2}, x^{1})
\end{equation*}
form a global pseudo-orthonormal frame of $TH^{3}_{1}$
such that $\langle X_{2}, X_{2} \rangle =
\langle X_{1}, X_{1} \rangle =1$ and 
$\langle X_{1}, X_{2} \rangle = 
\langle X_{1}, X_{3} \rangle =
\langle X_{2}, X_{3} \rangle = 0$.

Now, recall that the Lie group 
\begin{align*}
 \mathrm{SU}(1, 1) & = \left\{
 A \in \mathrm{GL}(2, \mathbb{C}) \mid
 {}^{t}AI_{1}\bar{A} = I_{1}, \ \det A = 1 \right\}  \\[0.1cm]
 & = \left\{ \begin{pmatrix}
  a  &  \bar{b}  \\
  b  &  \bar{a}
  \end{pmatrix} \biggm|
  a, b \in \mathbb{C}, \ |a|^{2} - |b|^{2} = 1 \right\},
\end{align*}
where $I_{1} = \displaystyle\begin{pmatrix}
 1  &  0  \\
 0  &  -1  
\end{pmatrix}$,
has the Lie algebra
\begin{align*}
 \mathfrak{su}(1, 1) & = \left\{
 X \in \mathfrak{gl}(2, \mathbb{C})  \mid
 {}^{t}XI _{1} + I_{1}\bar{X} = 0, \ \operatorname{Tr} X = 0
 \right\}  \\
 & = \left\{ \begin{pmatrix}
 \sqrt{-1}x^{3}  &  x^{2} - \sqrt{-1}x^{1}  \\[0.1cm]
 x^{2} + \sqrt{-1}x^{1}  &  - \sqrt{-1}x^{3}
 \end{pmatrix}
 \Biggm| x^{1}, x^{2}, x^{3} \in \mathbb{R}
 \right \},
\end{align*}
which is identified with $\mathbb{R}^{3}_{1}$, equipped 
with the scalar product $\langle X, Y \rangle = (1/2)
\operatorname{Tr}(XY)$, so that
\begin{equation}\label{eqn:5.3}
 e_{1} = \begin{pmatrix}
 0  &  - \sqrt{-1}\,  \\
 \sqrt{-1}  &  0
 \end{pmatrix},
 \quad  e_{2} = \begin{pmatrix}
 0  &  1  \\
 1  &  0
 \end{pmatrix},
 \quad  e_{3} = \begin{pmatrix}
 \sqrt{-1}  &  0  \\
 0  &  - \sqrt{-1}
 \end{pmatrix}
\end{equation}
form a pseudo-orthonormal basis of 
$\left(\mathfrak{su}(1, 1), \langle\ , \,\rangle\right)$.

Note that the anti-de Sitter $3$-space 
$H^{3}_{1}(1)$ 
is identified with $\mathrm{SU}(1, 1)$ under the map
\begin{equation}\label{eqn:5.4}
\begin{aligned}
 \psi : H^{3}_{1}(1)  &  \to \mathrm{SU}(1, 1),  \\
  x = (x^{1}, x^{2}, x^{3}, x^{4})  &  \mapsto
  A_{x} = \sqrt{-1}
   \begin{pmatrix}
   \bar{z}_{2}  &  - z_{1}  \\
   \bar{z}_{1}  &  - z_{2}
   \end{pmatrix}.
\end{aligned}
\end{equation}
Moreover, the adjoint representation of $\mathrm{SU}(1, 1)$
induces a covering homomorphism
\begin{equation}\label{eqn:5.5}
 \rho :  \mathrm{SU}(1, 1) \to \mathrm{SO}^{+}(1, 2),
\end{equation}
where $\mathrm{SO}^{+}(1, 2)$ is the restricted Lorentz group
with signature $(1, 2)$, that is, the identity component of 
the group of linear isometries $\mathrm{O}(1, 2)$ of 
$\mathbb{R}^{3}_{1}$.
Indeed, $\rho(A_{x})$ is defined as
\begin{equation*}
 \rho(A_{x}) : \mathfrak{su}(1, 1) \to \mathfrak{su}(1, 1),
 \quad  Y \mapsto \operatorname{Ad}(A_{x})Y
 = A_{x}YA_{x}^{-1},
\end{equation*}
and, with respect to the pseudo-orthonormal basis 
\eqref{eqn:5.3} of $\mathfrak{su}(1, 1)$,
the matrix representation of $\rho(A_{x})$ is given by
\begin{equation}\label{eqn:5.6}
\begin{aligned}
 \rho(A_{x}) 
 & =
 \begin{pmatrix}
  - \operatorname{Re}(z^{2}_{1} + \bar{z}^{2}_{2})
  &  - \operatorname{Im}(z^{2}_{1} - \bar{z}^{2}_{2})
  &  2\operatorname{Re}(z_{1}\bar{z}_{2})  \\[0.1cm]
  - \operatorname{Im}(z^{2}_{1} + \bar{z}^{2}_{2})
  &  \operatorname{Re}(z^{2}_{1} - \bar{z}^{2}_{2})
  &  2 \operatorname{Im}(z_{1}\bar{z}_{2})  \\[0.1cm]
  -2 \operatorname{Re}(z_{1}z_{2})  
  &  - 2 \operatorname{Im}(z_{1}z_{2})
  &  |z_{1}|^{2} + |z_{2}|^{2}
 \end{pmatrix}  \\[0.1cm]
  & = 
 \begin{pmatrix}
  A_{x}e_{1}A_{x}^{-1}  &  A_{x}e_{2}A_{x}^{-1}  &
  A_{x}e_{3}A_{x}^{-1}
  \end{pmatrix},
 \end{aligned}
\end{equation}
from which we easily see that the kernel of $\rho$ is 
$\{\pm \operatorname{Id}\}$.

The unit tangent bundle 
$\pi : T^{1}\mathbb{H}^{2}(c) \to \mathbb{H}^{2}(c)$
of the hyperbolic plane $\mathbb{H}^{2}(c)$ is defined to be
\begin{equation*}
\begin{aligned}
 T^{1}\mathbb{H}^{2}(c)  & =  \left\{ (x, v) \in 
 \mathbb{R}^{3}_{1} \times \mathbb{R}^{3}_{1} 
  \mid  x \in \mathbb{H}^{2}(c), 
  \ v \in T_{x}\mathbb{H}^{2}(c), \ \abs{v} = 1\right\}  \\
 & =  \left\{ (x, v) \in \mathbb{R}^{3}_{1} \times 
  \mathbb{R}^{3}_{1} \mid  \langle x, x \rangle = -1/c,
  \ \langle v, v \rangle = 1, \ \langle x, v \rangle = 0
  \right\}
\end{aligned}
\end{equation*}
with the canonical projection $\pi(x, v) = x$.
As in the spherical case in \S\ref{sect:2},
we may identify $T^{1}\mathbb{H}^{2}(c)$ with
$\mathrm{SO}^{+}(1, 2)$ by the diffeomorphism
\begin{equation}\label{eqn:5.7}
 \phi : \mathrm{SO}^{+}(1, 2) \to T^{1}\mathbb{H}^{2}(c),
 \quad  (c_{1} \ c_{2} \ c_{3}) \mapsto
 (c_{3}/\sqrt{c}, c_{1}).
\end{equation}
Finally, let $\iota$ be the homothety defined by
\[
 \iota : H^{3}_{1}(c/4) \to H^{3}_{1}(1),
 \quad  2x/\sqrt{c} \mapsto x.
\]
Then, it is immediate from \eqref{eqn:5.1} through \eqref{eqn:5.7} 
that the composition of the covering map
\begin{equation}\label{eqn:5.8}
 F = \phi \circ \rho \circ \psi \circ \iota : H^{3}_{1}(c/4)
 \to T^{1}\mathbb{H}^{2}(c)
\end{equation}
with the canonical projection 
$\pi : T^{1}\mathbb{H}^{2}(c) \to \mathbb{H}^{2}(c)$
yields the hyperbolic Hopf map 
$H : H^{3}_{1}(c/4) \to \mathbb{H}^{2}(c)$ of \eqref{eqn:5.1}.
Indeed, for each $2x/\sqrt{c} \in H^{3}_{1}(c/4)$
we have
\begin{equation}\label{eqn:5.9}
 F(2x/\sqrt{c}) = (\tilde{x}, e) \in T^{1}\mathbb{H}^{2}(c),
\end{equation}
where $\tilde{x} = (1/\sqrt{c})A_{x}e_{3}A_{x}^{-1}$
and $e = A_{x}e_{1}A_{x}^{-1}$, so that 
\begin{equation*}
 \pi \circ F(2x/\sqrt{c}) = (1/\sqrt{c})\left( 2z_{1}\bar{z}_{2},
 |z_{1}|^{2} + |z_{2}|^{2} \right) = H(z).
\end{equation*}

\subsection{Differentials of maps}\label{sect:5.2}
The differentials of maps appeared in \eqref{eqn:5.8} can be 
computed in the same way as in \S \ref{sect:3.1},
so we only remark on the following.

1)  Given $x \in H^{3}_{1}(1)$, the differential of $\psi$
in \eqref{eqn:5.4}
\[
 d\psi_{x} : T_{x}H^{3}_{1}(1) \to 
 T_{\psi(x)}(\mathrm{SU}(1, 1)) = 
 A_{x}\cdot \mathfrak{su}(1, 1)
\]
is given by
\begin{equation}\label{eqn:5.10}
 d\psi_{x}(X_{3}(x)) = A_{x}e_{3},  \quad
 d\psi_{x}(X_{2}(x)) = A_{x}e_{2},  \quad
  d\psi_{x}(X_{1}(x)) = A_{x}e_{1}.
\end{equation}

2)  The differential of $\rho$ in \eqref{eqn:5.5}
\[
 d\rho_{A_{x}} : T_{A_{x}}(\mathrm{SU}(1, 1)) = 
 A_{x} \cdot \mathfrak{su}(1, 1) \to T_{\rho(A_{x})}
 \mathrm{SO}^{+}(1, 2) = \rho(A_{x}) \cdot
 \mathfrak{so}(1, 2)
\]
is a linear map sending
\[
 A_{x}Y \mapsto d\rho_{A_{x}}(A_{x}Y) =
 \rho(A_{x}) \circ \operatorname{ad}(Y),
\]
so that we have
\begin{equation}\label{eqn:5.11}
\begin{aligned}
 d\rho_{A_{x}}(A_{x}e_{3}) & =
 \begin{pmatrix} 
  2A_{x}e_{2}A_{x}^{-1}  &  -2A_{x}e_{1}A_{x}^{-1}
   &  0
   \end{pmatrix},  \\
 d\rho_{A_{x}}(A_{x}e_{2}) & =
  \begin{pmatrix} 
   2A_{x}e_{3}A_{x}^{-1}  &  0
   &  2A_{x}e_{1}A_{x}^{-1}
   \end{pmatrix},  \\
 d\rho_{A_{x}}(A_{x}e_{1}) & =
  \begin{pmatrix} 
  0  &  -2A_{x}e_{3}A_{x}^{-1}
  & -2A_{x}e_{2}A_{x}^{-1}
  \end{pmatrix},
\end{aligned}
\end{equation}
since $\operatorname{ad}(e_{1})(e_{1}) = 0,
\ \operatorname{ad}(e_{1})(e_{2}) = -2e_{3},
\ \operatorname{ad}(e_{1})(e_{3}) = -2e_{2},
\ \operatorname{ad}(e_{2})(e_{3}) = 2e_{1}$
for the pseudo-orthonormal basis \eqref{eqn:5.3} of
$\mathfrak{su}(1, 1)$.

3)  Combining \eqref{eqn:5.10} with \eqref{eqn:5.11}
and taking into account the differentials of the diffeomorphism
$\phi$ and the homothety $\iota$, we find that the
differential of $F$ in \eqref{eqn:5.8}
\[
 dF_{x} : T_{x}H^{3}_{1}(c/4) \to
 T_{F(x)}(T^{1}\mathbb{H}^{2}(c))
\]
is determined as
\begin{equation}\label{eqn:5.12}
\begin{aligned}
 dF_{x}(2X_{3}(x)/\sqrt{c})  & = 
 (0, \ 2 A_{x}e_{2}A_{x}^{-1}) = \tilde{e}_{3},  \\
 dF_{x}(2X_{2}(x)/\sqrt{c})  & = 
 (2A_{x}e_{1}A_{x}^{-1}/\sqrt{c}, 
 \ 2 A_{x}e_{3}A_{x}^{-1}) = \tilde{e}_{2},  \\
 dF_{x}(2X_{1}(x)/\sqrt{c})  & = 
 (-2 A_{x}e_{2}A_{x}^{-1}/\sqrt{c}, \ 0) = \tilde{e}_{1}
\end{aligned}
\end{equation}
for each $x \in H^{3}_{1}(c/4)$.

\subsection{Lifts to the unit tangent bundle}
Recall that the unit tangent bundle $T^{1}\mathbb{H}^{2}(c)$
is a $3$-dimensional hypersurface of $T\mathbb{H}^{2}(c)$.
As in the spherical case in \S \ref{sect:3.2},
denoting by $X^{h}$ (resp.\ $Y^{v}$) the horizontal 
(resp.\ vertical) lift of $X$ (resp.\ $Y$),
we see that at $(x, e) \in T^{1}\mathbb{H}^{2}(c)$
the tangent space of the tangent bundle $T\mathbb{H}^{2}(c)$
is written as
\[
 T_{(x,e)}(T\hn^{2}(c)) = \left\{ X^{h} + Y^{v} \mid 
 X, Y \in T_{x}\hn^{2}(c) \right\},
\]
whereas that of the unit tangent bundle $T^{1}\mathbb{H}^{2}(c)$
is given by
\[
 T_{(x,e)}(T^{1}\hn^{2}(c)) = \left\{ X^{h} + Y^{v}
 \mid  X, Y \in T_{x}\hn^{2}(c), \ \langle Y, e \rangle = 0 \right\}.
\]

Recalling \eqref{eqn:5.9}, we set
\[
 e = A_{x}e_{1}A_{x}^{-1},  \quad
 f = - A_{x}e_{2}A_{x}^{-1},
\]
and $\tilde{x} = (1/\sqrt{c})A_{x}e_{3}A_{x}^{-1}$.
Then $(\tilde{x}, f) \in T^{1}\mathbb{H}^{2}(c)$ and
$\langle f, e \rangle = 0$, so that
\begin{equation*}
 T_{(\tilde{x}, e)}(T^{1}\mathbb{H}^{2}(c)) =
 \operatorname{Span}\left\{ e^{h}, f^{h}, f^{v} \right\}.
\end{equation*}
Furthermore, we have the following
\begin{proposition}\label{prop:4}
Let \,$\tilde{x}$, $e$ and $f$ be as above.  Then
\begin{equation}\label{eqn:5.13}
 (\sqrt{c}/2) \tilde{e}_{2}= e^{h},  \quad  
 (\sqrt{c}/2)\tilde{e}_{1}=f^{h},  \quad  
 \tilde{e}_{3} = -2f^{v}.
\end{equation}
\end{proposition}
\begin{proof}
This can be seen in the same manner as in the proof of
Proposition~\ref{prop:2}, so we only remark on the
following for the sake of completeness.

For the horizontal lift $e^{h}$, 
we consider a geodesic $\gamma : I \to \hn^{2}(c)$ 
starting from $\tilde{x} \in \hn^{2}(c)$ with initial vector 
$e \in T^{1}_{\tilde{x}}\hn^{2}(c)$. 
Then the curve $\Gamma : I \to T\hn^{2}(c)$ given by 
$\Gamma(t) = (\gamma(t), v(t) = \dot{\gamma}(t))$ 
satisfies that $\Gamma(0) = (\tilde{x}, e)$ and
$\nabla_{\dot{\gamma}(t)} v(t) = 0$ for all $t \in I$. 
Since
\[
 \gamma(t) = \cosh(\sqrt{c}t)\tilde{x} +
 \sinh(\sqrt{c}t)(e/\sqrt{c}),
\]
we deduce that
\[
 e^{h} = \dot{\Gamma}(0) = (e, c\tilde{x}) = (\sqrt{c}/2)
 \tilde{e}_{2}.
\]

Similarly, for $f^{h}$, we take a geodesic
$\gamma : I \to \hn^{2}(c)$ defined by
\[
 \gamma(t) = \cosh(\sqrt{c}t)\tilde{x} + 
 \sinh(\sqrt{c}t)(f/\sqrt{c}),
\]
starting from $\tilde{x} \in \hn^{2}(c)$ with initial vector 
$f \in T_{\tilde{x}}^{1}\hn^{2}(c)$. 
Then the curve $\Gamma : I \to T\hn^{2}(c)$ given by 
$\Gamma(t) = (\gamma(t), v(t) = e)$ satisfies that
$\Gamma(0) = (\tilde{x},e)$ and 
$\nabla_{\dot{\gamma}(t)} v(t) = 0$ for all $t \in I$.
Hence
\[
 f^{h} = \dot{\Gamma}(0) = (f, 0) = (\sqrt{c}/2)
 \tilde{e}_{1}.
\]

To construct the vertical lift $f^{v}$, 
we now consider a curve $\gamma : I \to T\hn^{2}(c)$
defined by 
$\gamma(t) = (\tilde{x}, (\cos t)e + (\sin t)f)$.
Then $\gamma(t)$ is a curve along the fibre over $\tilde{x}$
and satisfies that $\gamma(0) = (\tilde{x}, e)$ and 
$\dot{\gamma}(0) = (0, f)$.
Hence $\tilde{e}_{3} = (0, -2f) \in \v_{(\tilde{x},e)} \subset 
T_{(\tilde{x}, e)}(T^{1}\hn^{2}(c))
\subset T_{(\tilde{x}, e)}(T\hn^{2}(c))$. 
Moreover, for the connection map we have
\begin{align*}
 K_{(\tilde{x}, e)} (-\tilde{e}_{3}/2)  & = 
 \left.\frac{d}{dt}\right|_{t=0} 
 (\exp_{\tilde{x}} \circ R_{- e} \circ\tau)(\gamma(t))  \\
& = \left.\frac{d}{dt}\right|_{t=0} 
\big[ \exp_{\tilde{x}} ((\cos t) -1) e + (\sin t)f)\big] .
\end{align*}
Noting that the geodesic of $\hn^{2}(c)$ starting from
$\tilde{x}$ with unit initial vector $v$ is given by
$\delta_{(\tilde{x}, v)}(s) = \cosh(\sqrt{c}s)\tilde{x} + 
\sinh(\sqrt{c}s)(v/\sqrt{c})$,
we then see
\begin{align*}
 & \exp_{\tilde{x}}((\cos t -1)e + (\sin t)f)   \\
 & \quad  =  \cosh(\sqrt{c}\,\theta(t))\tilde{x} + 
 \frac{\sinh(\sqrt{c}\,\theta(t))}{\sqrt{c}} 
 \left( \frac{(\cos t -1)e + (\sin t)f}{\theta(t)} \right)
 = \Theta(t),
\end{align*}
where
\[
 \theta(t) = \left| (\cos t -1) e + (\sin t)f 
 \right|_{\mathbb{R}^{3}_{1}}
= \sqrt{2(1 - \cos t)}.
\]
Therefore we obtain
\begin{equation*}
 K_{(\tilde{x}, e)}(-\tilde{e}_{3}/2) = 
 \left.\frac{d}{dt}\right|_{t=0} \Theta(t) = f,
\end{equation*}
which shows that $\tilde{e}_{3} = -2f^{v}$.
\end{proof}

\subsection{Indefinite generalized Cheeger-Gromoll metrics}
\label{sect:5.4}
We extend the notion of the generalized Cheeger-Gromoll metric 
$h_{m,r}$ defined in \S \ref{sect:3.3} to admit indefinite ones.

More specifically, for the hyperbolic plane $\mathbb{H}^{2}(c)$,
we define on its tangent bundle $T\mathbb{H}^{2}(c)$ 
the {\em indefinite generalized Cheeger-Gromoll metric} $h_{m,r}$
as follows.
Given $m \in \mathbb{R}$ and $r \geq 0$, 
we set on each tangent space $T_{(x, e)}(T\mathbb{H}^{2}(c))$
\begin{equation}\label{eqn:5.14}
\begin{aligned}
 h_{m,r} (X^{h} , Y^{h} ) & = \langle X , Y \rangle,
 \quad  h_{m,r} (X^{h} , Y^{v} ) = 0,  \\
 h_{m,r} (X^{v} , Y^{v} ) & = - \omega^{m} 
 (\langle X , Y \rangle + r \langle X , e \rangle\langle  Y, 
 e \rangle),
\end{aligned}
\end{equation}
where  $X, Y \in T_{x}\hn^{2}(c)$ and $\omega = 1/(1+ |e|^{2})$.
It should be noted that, equipped with $h_{m,r}$ on
$T\mathbb{H}^{2}(c)$ and the canonical metric 
$\langle \ , \,\rangle$ on $\mathbb{H}^{2}(c)$, 
the canonical projection
$\pi : T\mathbb{H}^{2}(c) \to \mathbb{H}^{2}(c)$ yields
a submersion which is isometric on horizontal directions.
Moreover, when $(x, e) \in T^{1}\hn^{2}(c)$, 
this metric restricts on $T_{(x,e)}(T^{1}\hn^{2}(c))$ to
\begin{equation}\label{eqn:5.15}
\begin{aligned}
 h_{m,r} (X^{h}, Y^{h}) & = \langle X , Y \rangle,
 \quad  h_{m,r} (X^{h}, Y^{v}) = 0,  \\
 h_{m,r} (X^{v}, Y^{v}) & = - \frac{1}{2^{m}} 
 \langle X , Y \rangle.
\end{aligned}
\end{equation}
Note that the parameter $r$ disappears when restricted to 
the unit tangent bundle, and $h_{m,r}$ has a negative signature
on vertical directions.

With these understood, the proof of Theorem~\ref{thm:2}
is immediate.
Indeed, if we choose $m = \log_{2} c$, then, it follows from 
\eqref{eqn:5.12} and \eqref{eqn:5.13} together with
\eqref{eqn:5.15} that
\begin{align*}
 & h_{m,r}((\sqrt{c}/2)\tilde{e}_{1}, (\sqrt{c}/2)\tilde{e}_{1}) 
 = h_{m,r}(f^{h} , f^{h} ) = \langle f, f \rangle =1,   \\[0.1cm]
 & h_{m,r}((\sqrt{c}/2)\tilde{e}_{2}, (\sqrt{c}/2)\tilde{e}_{2}) 
 = h_{m,r}(e^{h} , e^{h} ) = \langle e, e \rangle = 1,  \\[0.1cm]
 & h_{m,r}((\sqrt{c}/2)\tilde{e}_{1}, (\sqrt{c}/2)\tilde{e}_{2}) 
 = h_{m,r}(f^{h}, e^{h}) = \langle f, e \rangle= 0,   \\[0.1cm]
 & h_{m,r}((\sqrt{c}/2)\tilde{e}_{2}, (\sqrt{c}/2)\tilde{e}_{3}) 
 = - h_{m,r}(e^{h}, \sqrt{c}f^{v}) = 0,  \\[0.1cm]
 & h_{m,r}((\sqrt{c}/2)\tilde{e}_{1}, (\sqrt{c}/2)\tilde{e}_{3}) 
 = - h_{m,r}(f^{h}, \sqrt{c}f^{v}) = 0,
\end{align*}
and
\begin{equation*}
 h_{m,r}((\sqrt{c}/2)\tilde{e}_{3}, (\sqrt{c}/2)\tilde{e}_{3}) 
 = h_{m,r}(-\sqrt{c}f^{v}, -\sqrt{c}f^v) 
 = - \frac{c}{2^{m}}\langle f , f \rangle = -1.
\end{equation*}
Consequently, the covering map 
$F : H^{3}_{1}(c/4) \to T^{1}\mathbb{H}^{2}(c)$ defined
by \eqref{eqn:5.8} gives rise to an isometric immersion from
$(H^{3}_{1}(c/4), g_{\mathrm{can}})$ to
$(T^{1}\mathbb{H}^{2}(c), h_{m,r})$ for $m = \log_{2} c$
and $r \geq 0$.


\bigskip

\begin{thebibliography}{10}
%
\bibitem{BLW1}
M.~Benyounes, E.~Loubeau and C.~M.~Wood,
\textit{Harmonic sections of Riemannian vector bundles,
and metrics of Cheeger-Gromoll type},
\newblock Diff. Geom Appl. \textbf{25} (2007), 322--334.
%
\bibitem{BLW2}
M.~Benyounes, E.~Loubeau and C.~M.~Wood,
\textit{The geometry of generalized Cheeger-Gromoll metrics},
\newblock Tokyo J.\ Math. \textbf{32} (2009), 287--312.
%
\bibitem{Berger}
M.~Berger,
\textit{Les vari\'{e}t\'{e}s riemanniennes homog\`{e}nes normales 
simplement connexes \`{a} courbure strictement positive},
\newblock Ann.\ Scuola Norm.\ Sup. Pisa (3) \textbf{15} (1961), 
179--246.
%
\bibitem{Milnor}
J.~Milnor,
\textit{Curvature of left-invariant metrics on Lie groups},
\newblock Adv. Math. \textbf{21} (1976), 293--329.
%
\bibitem{KS}
W.~Klingenberg and S.~Sasaki,
\textit{On the tangent sphere bundle of a $2$-sphere},
\newblock Tohoku Math. J. \textbf{27} (1975), 49--56.
%
\bibitem{MT}
E.~Musso and F.~Tricerri,
\textit{Riemannian metrics on tangent bundles},
\newblock Ann.\ Mat.\ Pura Appl. \textbf{150} (1988), 1--19.
%
\bibitem{O'Neill}
B.~O'Neill,
Semi-Riemannian Geometry, with applications to relativity,
\newblock Pure and Applied Mathematics, 103,
Academic Press, Inc., New York, 1983.
%
\bibitem{Sasaki}
S.~Sasaki,
\textit{On the differential geometry of tangent bundles of
Riemannian manifolds},
\newblock Tohoku Math. J. \textbf{10} (1958), 338--354.
%
\end{thebibliography}
\end{document}